\documentclass[11pt,reqno]{amsart}
\usepackage{amsmath,amsfonts,amssymb,mathrsfs,graphicx,subfigure}
\usepackage{color,psfrag,graphicx}
\usepackage{algorithm,algpseudocode,enumerate}
\usepackage[subnum]{cases}
\usepackage{lineno,empheq}
\usepackage[normalem]{ulem}

\newtheorem{theorem}{Theorem}[section]
\newtheorem{remark}[theorem]{Remark}
\newtheorem{lemma}[theorem]{Lemma}

\usepackage[pdftex,colorlinks,bookmarksopen,bookmarksnumbered,citecolor=red,urlcolor=red]{hyperref}

\newcommand{\tu}{\tilde{u}}
\newcommand{\bu}{\bar{u}}

\newcommand{\e}{\mathbb{E}}

\def\bn{\boldsymbol{n}}

\def\leq{\leqslant}
\def\geq{\geqslant}

\def\Lo{\mathcal{L}}

\newcommand{\eps}{\varepsilon}

 \newcommand{\set}[1]{ \left\{#1\right\}}

\newcommand{\rd}{\mathrm{d}}

\newcommand{\dom}{{D}}

\newcommand{\Od}{\mathcal{O}}
\newcommand{\ro}{\mathrm{\bf{ot}}}
\newcommand{\ri}{\mathrm{\bf{in}}}

\newcommand{\re}{\mathrm{e}}

\renewcommand{\vec}[1]{\boldsymbol{#1}}
\DeclareMathAlphabet{\mathsfsl}{OT1}{cmss}{m}{sl}

\renewcommand{\keywords}[1]{\textbf{\textit{Keywords: }}#1}

\begin{document}

\begin{center}
{\Large  Asymptotic Expansion with Boundary Layer Analysis for Strongly Anisotropic Elliptic Equations}

\
\

Ling Lin \footnote{email: linglin@cityu.edu.hk. Corresponding author.} and
Xiang Zhou \footnote{email: xiang.zhou@cityu.edu.hk. The research of XZ  was supported by the grants from the Research Grants Council of the Hong Kong Special Administrative Region, China (Project No. CityU 11304314,    11304715
and 11337216).  }
\par
\
\par

Department of Mathematics
\par
City University of Hong Kong\par
Tat Chee Ave, Kowloon \par Hong Kong SAR

\end{center}

\date{\today}
\section*{abstract}
In this article, we derive the asymptotic expansion, up to an arbitrary order in theory, for  the solution of a two-dimensional elliptic equation
with strongly anisotropic    diffusion coefficients along different directions, subject to the Neumann boundary condition
and the  Dirichlet boundary condition on specific parts of the domain boundary, respectively.
The ill-posedness arising from the Neumann boundary condition in the strongly anisotropic diffusion limit
is handled by the decomposition of the solution into a  mean part and a fluctuation part.
The boundary layer analysis due to the Dirichlet boundary condition is conducted for each order in the expansion for the fluctuation part. Our results suggest that the leading order
is the combination of the mean part and the composite approximation of the fluctuation part for the general
Dirichlet boundary condition.

\subjclass[]{}
\keywords{strongly anisotropic elliptic equation, boundary layer analysis, matched asymptotic analysis.}


\section{Introduction}

The strongly anisotropic elliptic problem we consider in this article is the following
equation  imposed in the  domain $\dom=(0,1)\times(0,1)$ with mixed Dirichlet--Neumann boundary conditions:
\begin{equation}\label{eqn:PDE}
\begin{cases}
 -\eps^{-2} \partial_x^2 u_\eps(x,y)- \partial_y^2 u_\eps(x,y) =f(x,y) & \quad \text{in } \dom ,\\
 \partial_x u_\eps(0,y)=\partial_x u_\eps(1,y)=0,  \quad &0\leq y\leq 1,\\
  u_\eps(x,0)=\phi_0(x),   \quad  u_\eps(x,1)=\phi_1(x),  &0\leq x\leq 1.
\end{cases}
\end{equation}
The special feature for this equation is that   $\eps$ is a small positive number,
 that is, the diffusion coefficient along the $x$-direction is very large.
Here the Neumann boundary conditions are imposed on the left and right boundaries
and the Dirichlet boundary conditions are imposed on the top and bottom boundaries
of the rectangular domain.

The equation \eqref{eqn:PDE}  belongs to  a large class
of diffusion models with the strongly anisotropic diffusion coefficients
 from many applications, e.g.,  image
processing \cite{wang2008},  flows in porous
media \cite{Ashby1999,Hou1997},  semiconductor modeling \cite{Manku1993},
heat conduction in fusion plasmas \cite{VanEs2014JCP}, and so on.
Note that here we use $\eps^2$ rather than $\eps$ in the diffusion coefficients
for the reason we will mention later.
Furthermore, in this simplified model \eqref{eqn:PDE}, the line field parallels to the $x$-axis.
In more realistic models, the line field may not be so simple and could be a closed loop.
Our main interest   is to
examine the asymptotic behaviors of the solution $u_\eps$ such as
$u_\eps \sim u_0+\eps u_1 + \eps^2 u_2+\cdots$.  To illustrate the main ideas,
the equation \eqref{eqn:PDE} serves a good model which can simply  lots of
technical calculations.

We first briefly review the existing works on the limit  of the solution as $\eps \downarrow 0$.
Mathematically,  the Neumann boundary conditions yield an ill-posed limiting problem
as taking the formal limit   $\eps \downarrow 0$ in the original problem \eqref{eqn:PDE}
 \cite{Degond2012CMS}.  The traditional numerical methods for the elliptic equations, such as
 the standard five-point scheme, suffer from the large condition numbers for tiny values of $\eps$.
There have been a lot of efforts focusing on the numerical methods for the strongly anisotropic elliptic problems.
In particular,
the class of asymptotic preserving method was developed recently by P. Degond et al.
 in a series of papers, e.g.,  \cite{Degond2010MMS,Degond2010JCP,Degond2012JCP,Degond2012CMS}.
Their  main idea is to decompose the solution into two parts, a  {\it mean}  part
along the  strongly diffusive direction and a fluctuation part, and then they reformulate the original equation into a
coupled system of the equations for these two parts.
More recently, \cite{Tang2016} proposed a new innovative  approach   to  replace one of the Neumann boundary condition
by the integration of the original equation along the field line.
 Consequently, the singular terms can be replaced by some regular terms, which  yields a
well-posed limiting problem.

   Heuristically,  the vanishing $\eps$ means  that
the variation of the solution  $u_\eps(x,y) $ along the $x$-direction is very slow, which then    suggests that
the limiting solution $  u_0=\lim_{\eps\downarrow 0}u_\eps$
(defined in certain sense) is constant along the $x$-direction, i.e., a function of the $y$ variable only.
 However, this would be inconsistent with the Dirichlet boundary conditions
in \eqref{eqn:PDE} if either of $\phi_0$ and $\phi_1$ were nonconstant.
In other words, a function of only the $y$ variable (such as the so-called mean part)  is, in general, not capable of describing
the limiting solution in the whole domain because there exist boundary layers near each nonconstant Dirichlet boundary.
Inside these boundary layers,
the function of only the $y$ variable has to be corrected to match the nonconstant Dirichlet boundary conditions.
It is noteworthy that  for  existing  numerical examples  presented in  previous  works such as \cite{Degond2012JCP, Tang2016}, as far as the authors know,  the Dirichlet boundary conditions are always homogeneous,
and so there are no boundary layers. Actually, those numerical examples are intentionally constructed
by choosing a true solution without boundary layers first and then defining the force term  $f$ accordingly.
However, as we argued in the above,
the emergence of the boundary layer is generic.
This   phenomena certainly makes our asymptotic analysis more complicated.

The existence of the boundary layers for the strongly anisotropic elliptic problem \eqref{eqn:PDE}
can also be easily seen from its probabilistic  interpretation which
 is connected to a special type of the  stochastic  fast-slow dynamics.
 The random walk model corresponding to the elliptic problem \eqref{eqn:PDE} is very simple.
Let $(X_t, Y_t)$ be the position of a particle in $\dom$ satisfying the following
stochastic differential equation
\begin{equation}\label{SDE}
\begin{cases}
&\rd X_t = \eps^{-1}\, \rd W_t, \\
&\rd Y_t =   \rd B_t, \\
  \end{cases}
\end{equation}
subject to the reflection boundary condition on the left and right boundaries ($x=0,1$) and
  the absorbing boundary condition on the top and bottom boundaries ($y=0,1$).
Here $W_t$ and $B_t$ are two independent (standard) Brownian motions.
By the Feymann-Kac formula \cite{Oksendal}, the solution to \eqref{eqn:PDE} is represented by
\begin{equation}\label{eqn:ueprobrep}
u_\eps(x,y)
=\e\biggl[\phi_{Y_\tau}(X_{\tau})+\frac12\int_0^{\tau} f(X_{t},Y_{t})\,\mathrm{d} t ~\bigg|~ (X_0, Y_0)=(x, y)\biggr],
\end{equation}
where $\tau=\inf\set{t>0: Y_t=0 \text{ or }1}$ is the absorption time of the $Y$ process to the Dirichlet boundaries.
The solution to \eqref{SDE} is straightforward:   $X_t = X_0 +  \eps^{-1} W_{t}\overset{\rd}{=}
X_0 +    W_{t/\eps^2}$ and $Y_t =Y_0 + B_t$. (``$\overset{\rd}{=}
$" means the equality   in the sense of distribution.)
So $X_t$ is a fast process and $Y_t$ is a slow process.
The particle randomly moves  drastically fast with the speed at the order $\Od(\eps^{-1})$  along the $x$-direction while
at a normal speed at the order $\Od(1)$ in the $y$-direction.
By the averaging principle,
 the leading
order dynamics as the limit  of $\eps\downarrow0$ is the expectation of the slow dynamics for  $Y_t$  with respect to the invariant measure
of the fast variable $X_t$, which is a uniform distribution here. Thus  the expectation of the integral part in \eqref{eqn:ueprobrep}
should have a limit independent of the $x$ variable, which is exactly the so called {\it mean} part in \cite{Degond2010MMS}.
But for the expectation of the first term  in \eqref{eqn:ueprobrep},
it depends on the distribution of the absorbing point $(X_\tau, Y_\tau)$.
If the starting position $(x, y)$ is away from the absorbing boundary,
then the absorbing time $\tau$ is sufficient large compared to the $\Od(\eps)$ relaxation time  to the equilibrium  in the $x$-direction,
so that the averaging principle still holds, and thus the limit of \eqref{eqn:ueprobrep} is a function of the variable $y$ only.
 However, the averaging principle   breaks down
if the initial position $(x, y)$ is too close to the absorbing boundary
so that $\tau$ would be too short to allow the fast dynamics to relax to the equilibrium.
It is easy to see that this occurs if the distance to the boundary  is $\Od(\eps)$,
thus the thickness of the boundary layers
around $y=0,1$ is $\Od(\eps)$.

Our main motivation is to give a more detailed understanding of the above probabilistic picture
by the tool of asymptotic analysis.  The goal is  to
 derive  a series of  approximate functions  to  the solution $u_\eps$  up to an arbitrary order as $\eps\downarrow0$.
In this note, we shall consider the general Dirichlet boundary conditions in \eqref{eqn:PDE}.
This means that  we should include  the boundary layer analysis  in our asymptotic expansion.
We shall show that  each asymptotic term   $u_0, u_1,u_2,\ldots$   exhibits  the boundary layer effect.
In particular, the leading order $u_0(x,y)$  is not simply the mean part  $\bu(y)=\int_0^1 u_\eps(x,y)\,\rd x$.
 To  attack the ill-posedness arising from the Neumann boundary conditions,
 we utilize the strategy of decomposing the solution into  a mean part
 and a fluctuation part \cite{Degond2010MMS,Degond2010JCP,Degond2012JCP,Degond2012CMS}; to deal with
 the boundary layers originating from the nonconstant Dirichlet boundary conditions,
 we adopt  the Van Dyke's method of matched asymptotic
expansions \cite{VanDyke:1975}. Thus the outer expansion and the inner expansion are both conducted.
The series in the outer expansion are described by the $y$-parametrized one-dimensional Neumann boundary value problem
in the $x$ variable,
while the series in the inner expansion are in the form of the two-dimensional  elliptic  equations
which are solved with  the aid of the Fourier series.


The rest of the paper is organized as follows.
Section \ref{sec:exp} presents  our main result of the
asymptotic expansion.
Section \ref{sec:proof} gives a rigorous proof
of our formal expansion.
Section \ref{sec:num} shows the numerical results
to validate the convergence  order in $\eps$
and demonstrate the computational efficiency.
The last section contains our concluding discussion.
\bigskip

\section{Asymptotic Result}
\label{sec:exp}

%

\subsection{Decomposing the solution into the mean value and the fluctuation}
For the solution $u_\eps$ to \eqref{eqn:PDE},
we introduce the mean part $\bu$ along the line field, i.e., the $x$-coordinate
\[
\bu(y):=\int_0^1 u_\eps(x,y)\,\rd x,
\]
and denote the residual as the fluctuation part $\tu_\eps$,
\[
\tu_\eps:={u_\eps-\bu}.
\]
Then by integrating both sides of the equation \eqref{eqn:PDE} with respect to $x$ over $[0,~1]$,
we obtain
\begin{equation}\label{eqn:buPDE}
\begin{cases}
& -\bu''(y)=\bar{f}(y), \quad \text{in } (0, ~1), \\
& \bu(0)=\bar{\phi}_{0}, \quad \bu(1)=\bar{\phi}_{1},
\end{cases}
\end{equation}
where
\[
\bar{f}(y)=\int_0^1 f(x,y)\,\rd x, \quad \bar{\phi}_{0}=\int_0^1 \phi_0(x)\,\rd x, \quad \bar{\phi}_{1}=\int_0^1 \phi_1(x)\,\rd x.
\]
Clearly, \eqref{eqn:buPDE} is a well-posed linear two-point boundary value problem, and $\bu$ can be solved uniquely.
Formally,
\begin{equation}
\label{UB}
\bu(y)=y\biggl(\int_0^1\int_0^z \bar{f}(t)\,\rd t\rd z-\bar{\phi}_{0}+\bar{\phi}_{1}\biggr)-\int_0^y\int_0^z\bar{f}(t)\,\rd t\rd z+\bar{\phi}_{0}.
\end{equation}

Subtracting \eqref{eqn:buPDE} from \eqref{eqn:PDE} yields the PDE  for the fluctuating part:
\begin{equation}\label{eqn:tuPDE}
\begin{cases}
 -\eps^{-2}\partial_x^2 \tu_\eps -\partial_y^2 \tu_\eps=
\tilde{f}, \quad &\text{in } \dom, \\
 \partial_x \tu_\eps(0,y)=\partial_x \tu_\eps(1,y)=0,  \quad &0\leq y\leq 1,\\
  \tu_\eps(x,0)=\tilde\phi_0(x),\quad \tu_\eps(x,1)=\tilde\phi_1(x),  &0\leq x\leq 1,
\end{cases}
\end{equation}
where
\[
\tilde{f}(x,y)=f(x,y)-\bar{f}(y), \quad \tilde\phi_0(x)=\phi_0(x)-\bar{\phi}_{0}, \quad
\tilde\phi_1(x)=\phi_1(x)-\bar{\phi}_{1}.
\]
Note that by construction, we have
\begin{equation}\label{eqn:intcond}
\int_0^1 \tu_\eps(x,y)\,\rd x=0,  \quad   \text{for } 0\leq y\leq 1,
\end{equation}
\[
\int_0^1 \tilde{f}(x,y)\,\rd x=0, \quad   \text{for } 0\leq y\leq 1,
\]
and
\[
\int_0^1\tilde\phi_{0}(x)\,\rd x=\int_0^1\tilde\phi_{1}(x)\,\rd x=0.
\]

\subsection{Asymptotic expansions of the fluctuation $\tu_\eps$}
Our main task is to seek an asymptotic expansion of the fluctuation $\tu_\eps$.
Formally,  as $\eps\downarrow 0$ in \eqref{eqn:tuPDE}
and \eqref{eqn:intcond}, the formal limit $\tu_0=\lim_{\eps\downarrow 0}\tu_\eps$ would
satisfy
\begin{equation}\label{eqn:tu0PDE}
\begin{cases}
 \partial_x^2 \tu_0=0, \quad &\text{in } \dom, \\
 \partial_x \tu_0(0,y)=\partial_x \tu_0(1,y)=0,  \quad &0\leq y\leq 1,\\
 \displaystyle\int_0^1 \tu_0(x,y)\,\rd x=0,  \quad   & 0\leq y\leq 1,\\
  \tu_0(x,0)=\tilde\phi_0(x),\quad \tu_0(x,1)=\tilde\phi_1(x),  &0\leq x\leq 1.
\end{cases}
\end{equation}
Clearly, this is an ill-posed problem unless  $\tilde\phi_0(x)\equiv0$
and $\tilde\phi_1(x)\equiv0$, since the first three equations in \eqref{eqn:tu0PDE}
yield $\tu_0\equiv0$.
This inconsistency implies that we have a singular perturbation problem
and anticipate the emergence of two boundary layer regions  near the Dirichlet boundaries $y=0$ and $y=1$ respectively.
We   apply the Van Dyke's method of matched asymptotic expansions \cite{VanDyke:1975} to tackle this problem, i.e.,
first separately solve the problem in
the inner regions within the boundary layers and in the outer region away from the boundary layers, and then match
them at the edges of the boundary layers.

\subsubsection{Outer expansion}

Assume the following outer expansion away from the Dirichlet boundaries $y=0$ and $y=1$:
\[
\tu_\eps^\ro(x,y)=\sum_{n=0}^\infty\eps^n\tu^\ro_n(x,y).
\]
Substituting this into the equation in \eqref{eqn:tuPDE} and equating coefficients, we obtain
\begin{equation}\label{eqn:tuo0PDE}
 -\partial_x^2 \tu^\ro_0=0,
\end{equation}
\begin{equation}\label{eqn:tuo1PDE}
 -\partial_x^2 \tu^\ro_1=0,
 \end{equation}
 \begin{equation}\label{eqn:tuo2PDE}
 -\partial_x^2 \tu^\ro_2=\partial_y^2\tu^\ro_0+\tilde{f},
\end{equation}
\begin{equation}\label{eqn:tuonPDE}
 -\partial_x^2 \tu^\ro_n=\partial_y^2\tu^\ro_{n-2},\quad n\geq3.
\end{equation}
{These equations   are   a set of parametric one dimensional  differential equations in the $x$ variable and
$y$ is in the role of parameters.}
The outer expansion solutions must also satisfy the Neumann boundary condition
as in \eqref{eqn:tuPDE} and the integral condition \eqref{eqn:intcond}, which gives for any $n$
\begin{equation}\label{eqn:tuonBC}
\begin{cases}
\partial_x \tu^\ro_n(0,y)=\partial_x \tu^\ro_n(1,y)=0,   \quad & 0\leq y\leq 1,\\
\displaystyle\int_0^1 \tu^\ro_n(x,y)\,\rd x=0, \quad & 0\leq y\leq 1.
\end{cases}
\end{equation}
Then each  $\tu^\ro_n$ is the unique solution to these Neumann problems due to the second condition in \eqref{eqn:tuonBC}.
We can solve $\tu^\ro_n$ recursively from  \eqref{eqn:tuo0PDE}$\sim$\eqref{eqn:tuonPDE} together with  \eqref{eqn:tuonBC}.
In particular, we have

\begin{align}
\tu_0^\ro(x,y) & \equiv0,  \label{OO}
\\
\tu_{n}^\ro(x,y)& \equiv0,\quad \text{for odd } n,
\nonumber\\
\tu_2^\ro(x,y)&=-\tilde{F}_2(x,y)+\tilde{F}_3(1,y). \label{2xy}
\end{align}
Here   $\tilde{F}_n$ is defined recursively as
\begin{equation}\label{Fnxy}
\tilde{F}_n(x,y)=
\begin{cases}
\displaystyle\int_0^x \tilde{F}_{n-1}(z,y)\,\rd z, \quad & n\geq 1,\\
\tilde{f}(x,y), & n=0.
\end{cases}
\end{equation}

\subsubsection{Inner expansion near $y=0$}

Next we explore  the inner solution near $y=0$ in terms of the stretched variable $\xi=y/{\eps}$
by assuming
\[
\tu_\eps^{\ri,0}(x,\xi)=\sum_{n=0}^\infty\eps^n\tu^{\ri,0}_n(x,\xi).
\]
In terms of $\xi$, the equation in \eqref{eqn:tuPDE} becomes
\begin{equation}\label{eqn:tuxiPDE}
-\eps^{-2}\partial_x^2 \tu_\eps(x,\xi) -\eps^{-2}\partial_\xi^2 \tu_\eps(x,\xi)=\tilde{f}(x,\eps\xi).
\end{equation}
Thus the inner expansion $\tu_\eps^{\ri,0}(x,\xi)$ near $y=0$   asymptotically satisfies the   equation \eqref{eqn:tuxiPDE}
with the boundary conditions
\[
\begin{cases}
& \partial_x \tu_\eps^{\ri,0}(x,\xi)=0,   \quad  x=0 \text{ or } 1,\\
& \tu^{\ri,0}_\eps(x,0)=\tilde\phi_0(x),
\end{cases}
\]
and the integral condition
\[
\int_0^1 \tu_\eps^{\ri,0}(x,\xi)\,\rd x=0, \quad \text{for } \xi\geq 0.
 \]
We can write  the  Taylor expansion of the fluctuation part of the external force:
\[
\tilde{f}(x,\eps\xi)=\sum_{n=0}^\infty\frac{\eps^n\xi^n}{n!}\partial_y^n \tilde{f}(x,0),
\]
then equate coefficients of the same powers of $\eps$ to obtain the following
two-dimensional elliptic  equations on the domain $(x,\xi)\in (0,1)\times(0,\infty)$:
\begin{equation}\label{eqn:tui00PDE}
\begin{cases}
& -\partial_x^2 \tu^{\ri,0}_0 -\partial_\xi^2 \tu^{\ri,0}_0=0, \quad 0<x<1, ~~\xi>0, \\
& \partial_x \tu^{\ri,0}_0(x,\xi)=0,   \quad  x=0 \text{ or } 1,\\
& \tu^{\ri,0}_0(x,0)=\tilde\phi_0(x),\\
& \displaystyle \int_0^1 \tu_0^{\ri,0}(x,\xi)\,\rd x=0, \quad \text{for } \xi\geq 0,
\end{cases}
\end{equation}
\begin{equation}\label{eqn:tui01PDE}
\begin{cases}
& -\partial_x^2 \tu^{\ri,0}_1 -\partial_\xi^2 \tu^{\ri,0}_1=0, \quad 0<x<1, ~~\xi>0, \\
& \partial_x \tu^{\ri,0}_1(x,\xi)=0,   \quad  x=0 \text{ or } 1,\\
& \tu^{\ri,0}_1(x,0)=0,\\
& \displaystyle\int_0^1 \tu_1^{\ri,0}(x,\xi)\,\rd x=0, \quad \text{for } \xi\geq 0,
\end{cases}
\end{equation}
and for $n\geq 2$,
\begin{equation}\label{eqn:tui0nPDE}
\begin{cases}
& -\partial_x^2 \tu^{\ri,0}_n -\partial_\xi^2 \tu^{\ri,0}_n=\dfrac{\xi^{n-2}}{(n-2)!}\partial_y^{n-2} \tilde{f}(x,0)
, \quad 0<x<1, ~~\xi>0,\\
& \partial_x \tu^{\ri,0}_n(x,\xi)=0,   \quad  x=0 \text{ or } 1,\\
& \tu^{\ri,0}_n(x,0)=0,\\
& \displaystyle \int_0^1 \tu_n^{\ri,0}(x,\xi)\,\rd x=0, \quad \text{for } \xi\geq 0.
\end{cases}
\end{equation}
These problems \eqref{eqn:tui00PDE},\eqref{eqn:tui01PDE} and \eqref{eqn:tui0nPDE}
do not have the uniqueness of the solutions even  the integral conditions $\int_0^1 \tu_n^{\ri,0}(x,\xi)\,\rd x=0$ are imposed.
The uniqueness comes from the matching to the outer solutions as we will show below.

We next solve $\tu^{\ri,0}_n$ by the method of separation of variables
because of the simple geometry of the domain.
We first work on the lowest order at $n=0$.
We look at the solutions that can be expanded into the form $\tu^{\ri,0}_0=\sum_kA_k(\xi)B_k(x)$.
The equation and boundary conditions in \eqref{eqn:tui00PDE} show that \[
B_k(x)=\cos(k\pi x), \quad k\geqslant 0,
\]
which form a complete orthogonal basis for the space $L^2([0,~1])$.
Thus we can expand  $\tu^{\ri,0}_0(\cdot,\xi)$ and  $\tilde{\phi}_0$ respectively
in terms of these  Fourier
cosine series:
\[
\tu^{\ri,0}_0(x,\xi)=\sum_{k=1}^\infty A_{k}(\xi)\cos(k\pi x),
\]
\[
\tilde{\phi}_0(x)=\sum_{k=1}^\infty \phi_{0,k}\cos(k\pi x),
\]
where the coefficients for $k\geq 1$ are
\[
A_k(\xi)=2\int_0^1 \tu^{\ri,0}_0(x,\xi)\cos(k\pi x)\,\rd x,
\]
\[
\tilde{\phi}_{0,k}=2\int_0^1 \tilde{\phi}_0(x)\cos(k\pi x)\,\rd x.
\]
Note that the terms for $k=0$ in these Fourier cosine series disappear
since
\[
A_0(\xi)=\int_0^1 \tu^{\ri,0}_0(x,\xi)\,\rd x=0,
\]
\[
\tilde{\phi}_{0,0}=\int_0^1 \tilde{\phi}_0(x)\,\rd x=0.
\]
Substituting these Fourier cosine expansions into \eqref{eqn:tui0nPDE}, we deduce that for each $k\geq 1$,
$A_k(\xi)$ satisfies
\[
\begin{cases}
&-A_k''+k^2\pi^2A_k=0,\\
&A_k(0)=\tilde{\phi}_{0,k}.
\end{cases}
\]
Hence we have  that  for $k\geqslant 1$,
\[
A_k(\xi)=(c_k+{\tilde{\phi}_{0,k}})\mathrm{e}^{-k\pi\xi}-c_k\mathrm{e}^{k\pi\xi},
\]
with the constants $c_k$ to be determined later by matching the outer and inner solutions.
\begin{remark}
We comment a bit on the possibility of generalizing the above calculations to a general line field.
One can  work  in the curvilinear coordinate of the field line
and obtain the equations for the outer expansions straightforwardly.
But since in general  the line field may not   match the Dirichlet boundary  like in our model \eqref{eqn:PDE},
then the boundary layer may not be a rectangular band with a uniform width $\eps$, and consequently,
the stretching variable $\xi$ would not be simply equal to $y/\eps$; the  geometric property  of the field line  and the Dirichlet boundary
should be incorporated to derive the equations for the inner expansion near the Dirichlet boundary.
\end{remark}

\subsubsection{Matching}
To determine the constants $c_k$'s in the first-term approximation of the boundary layer solution near $y=0$,
we make use of the essential point that the inner solution $\tu^{\ri,0}_0(x,\xi)$
 and the outer solution $\tu^{\ro}_0(x,y)$ should match on the boundary of the layer near $y=0$,
 that is,
 \[
 \lim_{\xi\rightarrow\infty} \tu^{\ri,0}_0(x,\xi)=\lim_{y\rightarrow0+} \tu^{\ro}_0(x,y)=0.
 \]
This gives
\[
c_k=0, \quad k\geq 1,
\]
and so
\[
A_k={\tilde{\phi}_{0,k}}\mathrm{e}^{-k\pi\xi}, \quad k\geq 1.
\]
Consequently,
\begin{equation} \label{in00}
\tu^{\ri,0}_0(x,\xi)=\sum_{k=1}^\infty \tilde{\phi}_{0,k}\mathrm{e}^{-k\pi\xi}\cos(k\pi x).
\end{equation}

\subsubsection{Inner expansion near $y=1$}
For the other Dirichlet boundary at $y=1$, we proceed in the exactly same way to derive the inner solution.
 Assume  the inner expansion near $y=1$ in terms of the stretched variable $\eta=(1-y)/\eps$,
\[
\tu^{\ri,1}_\eps(x,\eta)=\sum_{n=0}^\infty\eps^n\tu^{\ri,1}_n(x,\eta).
\]
In terms of $\eta$, the equation in \eqref{eqn:tuPDE} becomes
\[
-\eps^{-2}\partial_x^2 \tu_\eps(x,\eta) -\eps^{-2}\partial_\eta^2 \tu_\eps(x,\eta)=\tilde{f}(x,1-\eps\eta).
\]
Thus the inner expansion $\tu_\eps^{\ri,1}(x,\eta)$ near $y=1$ must asymptotically satisfy this equation
and the boundary and integral conditions
\[
\begin{cases}
& \partial_x \tu_\eps^{\ri,1}(x,\eta)=0,   \quad  x=0 \text{ or } 1,\\
& \tu^{\ri,1}_\eps(x,0)=0,\\
& \displaystyle \int_0^1 \tu^{\ri,1}_\eps(x,\eta)\,\rd x=0, \quad \text{for } \eta\geq 0.
\end{cases}
\]
Again, using Taylor expansion and then equating coefficients of like powers, we obtain
\[
\begin{cases}
& -\partial_x^2 \tu^{\ri,1}_0 -\partial_\eta^2 \tu^{\ri,1}_0=0, \quad 0<x<1, ~~\eta>0, \\
& \partial_x \tu^{\ri,1}_0(x,\eta)=0,   \quad  x=0 \text{ or } 1,\\
& \tu^{\ri,1}_0(x,0)=\tilde{\phi}_1,\\
& \displaystyle \int_0^1 \tu^{\ri,1}_0(x,\eta)\,\rd x=0, \quad \text{for } \eta\geq 0,
\end{cases}
\]
\[
\begin{cases}
& -\partial_x^2 \tu^{\ri,1}_1 -\partial_\eta^2 \tu^{\ri,1}_1=0, \quad 0<x<1, ~~\eta>0, \\
& \partial_x \tu^{\ri,1}_1(x,\eta)=0,   \quad  x=0 \text{ or } 1,\\
& \tu^{\ri,1}_1(x,0)=0,\\
& \displaystyle \int_0^1 \tu^{\ri,1}_1(x,\eta)\,\rd x=0, \quad \text{for } \eta\geq 0,
\end{cases}
\]
and for $n\geq 2$,
\[
\begin{cases}
& -\partial_x^2 \tu^{\ri,1}_n -\partial_\eta^2 \tu^{\ri,1}_n=\dfrac{(-1)^n\eta^{n-2}}{(n-2)!}\partial_y^{n-2} \tilde{f}(x,1), \quad 0<x<1, ~~\eta>0, \\
& \partial_x \tu^{\ri,1}_n(x,\eta)=0,   \quad  x=0 \text{ or } 1,\\
& \tu^{\ri,1}_n(x,0)=0,\\
& \displaystyle \int_0^1 \tu^{\ri,1}_n(x,\eta)\,\rd x=0, \quad \text{for } \eta\geq 0.
\end{cases}
\]
By the same token, we solve the above equations by Fourier cosine series and use the matching procedure to determine the constants,
then we obtain the lowest order
\begin{equation}\label{in10}
\tu^{\ri,1}_0(x,\eta)=\sum_{k=1}^\infty \tilde{\phi}_{1,k}\mathrm{e}^{-k\pi\eta}\cos(k\pi x),
\end{equation}
where
\[
\tilde{\phi}_{1,k}=2\int_0^1 \tilde{\phi}_1(x)\cos(k\pi x)\,\rd x.
\]

\subsubsection{Composite Expansion}
Now we can get the leading order term of $u_\eps$ which is valid on the whole domain.
Expressing all the three pieces of expansions in terms of $x$ and $y$, and combining them by adding them together and then
subtracting their common parts,
eventually we obtain the following composite approximation
by noting \eqref{UB}, \eqref{OO}, \eqref{in00} and \eqref{in10},

\begin{align}\label{eqn:0th order approx}
 u_\eps(x,y) =& \bu(y) + \tu_\eps(x,y)\nonumber \\
\sim &\bar{u}(y)+\bigl(\tu^{\ro}_0(x,y)+\tu^{\ri,0}_0(x,y/\eps)+\tu^{\ri,1}_0(x,(1-y)/\eps) \nonumber\\
&-\lim_{\xi\rightarrow\infty}\tu^{\ri,0}_0(x,\xi)-\lim_{\eta\rightarrow\infty}\tu^{\ri,1}_0(x,\eta)\bigr)  \nonumber\\
=&y\biggl(\int_0^1\int_0^z \bar{f}(t)\,\rd t\rd z-\bar{\phi}_{0}+\bar{\phi}_{1}\biggr)-\int_0^y\int_0^z\bar{f}(t)\,\rd t\rd z+\bar{\phi}_{0} \nonumber\\
&+\sum_{k=1}^\infty \Bigl(\tilde{\phi}_{0,k}\mathrm{e}^{-k\pi y/\eps}+\tilde{\phi}_{1,k}\mathrm{e}^{-k\pi (1-y)/\eps}\Bigr)
\cos(k\pi x)\nonumber\\
=:&u^{[0]}(x,y),
 \end{align}
 where $\bu(y)$ is the mean solution   given in \eqref{UB}.

\subsubsection{Higher order approximations}
Higher order approximations can be obtained similarly by the Van Dyke's method of matched asymptotic expansions \cite{VanDyke:1975}.
Let us compute the second order  expansion for demonstration for illustration.
To this end, we need to solve the next two orders in the inner expansion   near the boundaries $y=0,1$,
i.e.,
$\tu^{\ri,0}_n$ and $\tu^{\ri,1}_n$ for $n=1,2$.

Using the method of separation of variables again,
by expanding in the Fourier cosine series,
we obtain
\[
\tu^{\ri,0}_1(x,\xi)=\sum_{k=1}^\infty a_k(\mathrm{e}^{-k\pi\xi}-\mathrm{e}^{k\pi\xi})\cos(k\pi x),
\]
\[
\tu^{\ri,0}_2(x,\xi)=\sum_{k=1}^\infty \biggl[\frac{\tilde{f}_k(0)}{k^2\pi^2}+b_k\mathrm{e}^{k\pi\xi}-\Bigl(b_k+\frac{\tilde{f}_k(0)}{k^2\pi^2}\Bigr)
\mathrm{e}^{-k\pi\xi}\biggr]\cos(k\pi x),
\]
where
\[
\tilde{f}_k(y)=2\int_0^1 \tilde{f}(x,y)\cos(k\pi x)\,\rd x,
\]
$a_k$ and $b_k$ are undetermined constants.
Clearly, the $\mathrm{e}^{k\pi\xi}$ terms should disappear since they asymptotically blow up,
this implies that $a_k=0$ and $b_k=0$.
Hence, by guessing $a_k\equiv 0$ and $b_k\equiv 0$, we have the first three terms of the inner solution $\tu_\eps^{\ri,0}(x,\xi)$:
\[
\begin{split}
\tu_\eps^{\ri,0}(x,\xi)&=\tu_0^{\ri,0}(x,\xi)+\eps\tu_1^{\ri,0}(x,\xi)+\eps^2\tu_2^{\ri,0}(x,\xi)+\Od(\eps^3)\\
&=\sum_{k=1}^\infty \Bigl[\tilde{\phi}_{0,k}\mathrm{e}^{-k\pi\xi}+\eps^2\frac{\tilde{f}_k(0)}{k^2\pi^2}(1-\mathrm{e}^{-k\pi\xi})\Bigr]
\cos(k\pi x)+\Od(\eps^3).
\end{split}
\]
Note that   the outer expansion is
\[
\begin{split}
\tu_\eps^{\ro}(x,y)&=\tu_0^{\ro}(x,y)+\eps\tu_1^{\ro}(x,y)+\eps^2\tu_2^{\ro}(x,y)+\Od(\eps^3)\\
&=\eps^2\bigl(-\tilde{F}_2(x,y)+\tilde{F}_3(1,y)\bigr)+\Od(\eps^3).
\end{split}
\]
To see that these two expansions do match up to the given order,
 we write the outer expansion in terms of the inner variables
and    vice versa. By  dropping  the asymptotically negligible $\mathrm{e}^{-k\pi\xi}$  terms as $\xi\to\infty$,
 we have the  $\Od(\eps^3)$ approximations:
\begin{equation}\label{eqn:outinlim}
\bigl(\tu_\eps^{\ro}(x,\xi)\bigr)^{\ri,0}\approx\eps^2\bigl(-\tilde{F}_2(x,0)+\tilde{F}_3(1,0)\bigr),
\end{equation}
and
\begin{equation}\label{eqn:inoutlim}
\bigl(\tu_\eps^{\ri,0}(x,\xi)\bigr)^\ro\approx\sum_{k=1}^\infty \eps^2\frac{\tilde{f}_k(0)}{k^2\pi^2}
\cos(k\pi x).
\end{equation}
To show that matching (to this order) has been accomplished, we only need to check
that the right hand sides of \eqref{eqn:inoutlim} and  \eqref{eqn:outinlim} are equal.
In fact, we have for every $y$,  by  \eqref{eqn:tuonBC} and \eqref{2xy},
\[
\int_0^1 (-\tilde{F}_2(x,y)+\tilde{F}_3(1,y)\bigr)\,\rd x=0;
\]
from \eqref{Fnxy} and the integration by parts twice,  for $k\geq 1$,
\[
\begin{split}
&2\int_0^1\bigl(-\tilde{F}_2(x,y)+\tilde{F}_3(1,y)\bigr)\cos(k\pi x)\,\rd x
\\
=&\frac{2}{k\pi}\int_0^1\tilde{F}_1(x,y)\sin(k\pi x)\,\rd x\\
=&\frac{2}{k^2\pi^2}\int_0^1\tilde{f}(x,y)\cos(k\pi x)\,\rd x\\
=&\frac{\tilde{f}_k(y)}{k^2\pi^2}.
\end{split}
\]
Note that in the second equality, we also used the simple fact of the integral condition
\[
\tilde{F}_1(x,y)=\int_0^x \tilde{f}(z,y)\,\rd z=0, \quad \text{for }x=0,1.
\]

Analogously, we can solve
\[
\tu_1^{\ri,1}(x,\eta)\equiv0,
\]
\[
\tu_2^{\ri,1}(x,\eta)=\frac{\tilde{f}_k(1)}{k^2\pi^2}(1-\mathrm{e}^{-k\pi\eta}),
\]
and from
\[
\bigl(\tu_\eps^{\ro}(x,\xi)\bigr)^{\ri,1}\approx\eps^2\bigl(-\tilde{F}_2(x,1)+\tilde{F}_3(1,1)\bigr),
\]
\[
\bigl(\tu_\eps^{\ri,1}(x,\xi)\bigr)^\ro\approx\sum_{k=1}^\infty \eps^2\frac{\tilde{f}_k(1)}{k^2\pi^2}
\cos(k\pi x),
\]
we also see that matching (to this order) has been accomplished.

The last step is to combine the three expansions into a composite expansion
\[
\begin{split}
&u_\eps(x,y)
\sim \bar{u}(y)+\tu^{\ro}_0(x,y)+\eps\tu^{\ro}_1(x,y)+\eps^2\tu^{\ro}_2(x,y)\\
&\qquad+\tu^{\ri,0}_0(x,y/\eps)+\eps\tu^{\ri,0}_1(x,y/\eps)+\eps^2\tu^{\ri,0}_2(x,y/\eps)\\
&\qquad+\tu^{\ri,1}_0(x,(1-y)/\eps)+\eps\tu^{\ri,1}_1(x,(1-y)/\eps)+\eps^2\tu^{\ri,1}_2(x,(1-y)/\eps)\\
&\qquad-\bigl(\tu_\eps^{\ro}(x,\xi)\bigr)^{\ri,0}-\bigl(\tu_\eps^{\ro}(x,\xi)\bigr)^{\ri,1}\\
=&u^{[0]}(x,y)+\eps^2\sum_{k=1}^\infty \Bigl[\frac{\tilde{f}_k(y)}{k^2\pi^2}-\frac{\tilde{f}_k(0)}{k^2\pi^2}\mathrm{e}^{-k\pi y/\eps}-
\frac{\tilde{f}_k(1)}{k^2\pi^2}\mathrm{e}^{-k\pi(1-y)/\eps}\Bigr]\cos(k\pi x)\\
=:&u^{[2]}(x,y),
\end{split}
\]
where $u^{[0]}(x,y)$ is the leading order given in  \eqref{eqn:0th order approx}.

Clearly, using the above method, we may proceed to derive the asymptotic expansions of $u_\eps$
to any order, and in general, the form of the asymptotic expansion up to $\Od(\eps^{2n})$, $n=0,1,\dots$, is
\begin{equation}\label{eqn:2nth order approx}.
\begin{split}
u_\eps(x,y)
\sim & u^{[0]}(x,y)
+\sum_{m=1}^n\eps^{2m}\sum_{k=1}^\infty \biggl[\frac{\tilde{f}^{(2m-2)}_k(y)}{(k\pi)^{2m}}-\frac{\tilde{f}^{(2m-2)}_k(0)}{(k\pi)^{2m}}\mathrm{e}^{-k\pi y/\eps}\\
&-\frac{\tilde{f}^{(2m-2)}_k(1)}{(k\pi)^{2m}}\mathrm{e}^{-k\pi(1-y)/\eps}\biggr]\cos(k\pi x)\\
=:&u^{[2n]}(x,y).
\end{split}
\end{equation}
The justification of the approximation  orders of these asymptotic expansions will be given in the next section.

\subsection{Discussion}
For our   approximations \eqref{eqn:2nth order approx},
it is observed that the boundary layer terms
 would disappear if
\[
\tilde{\phi}_{0,k}=\tilde{\phi}_{1,k}=\tilde{f}^{(2m-2)}_k(0)=\tilde{f}^{(2m-2)}_k(1)=0, \quad \text{for all } 1\leq m\leq  n, \ k=1, 2, \cdots,
\]
i.e.,
\[
\phi_0(x),~ \phi_1(x), ~\partial_{y}^{2m}f(x,0), ~\partial_y^{2m}f(x,1), \quad  0\leq m\leq n-1,
\]
all happen to be constant functions independent of $x$.
Note that this condition is stronger than the homogeneous Dirichlet boundary conditions
since it also involves the even order normal derivatives up to $2n-2$ of the external force on the Dirichlet boundaries.
In this case free of the boundary layers, we only need to compute the outer expansions by solving the equations \eqref{eqn:tuo2PDE} and \eqref{eqn:tuonPDE} up to $2n$
to get the approximation $u^{[2n]}$.
Note that each equation in \eqref{eqn:tuo2PDE} and \eqref{eqn:tuonPDE} is actually
a system parametrized by the $y$ variable of independent   ordinary differential equations
in the $x$ variable  rather than a two-dimensional partial differential equation.
In computation, each equation can be solved by a numerical integrator in
parallel at all grid points of the $y$ variable which is now viewed as a parameter.
This can reduce  the  computational cost to  linear
scaling.
However, in general, the appearance of the boundary layers seems inevitable;
then many existing algorithms may need further improvements to approximate the solution
on the whole domain.

\section{Theoretical Justification}
\label{sec:proof}

Define the remainders in the asymptotic expansions of $u_\eps$:
\[
r_{2n}=u_\eps-u^{[2n]}, \quad n=0,1,\dots,
\]
where the approximations $u^{[2n]}$'s are given in \eqref{eqn:2nth order approx}.
We shall prove  the following estimates of the errors $r_{2n}$ in this section.
\begin{theorem}\label{thm:error estimate}
\[
\|r_{2n}\|_{\infty}=\Od(\eps^{2(n+1)}).
\]
\end{theorem}
To justify these estimates, we first establish a modified version of the maximum principle for the elliptic equation with mixed boundary
value conditions.
\begin{lemma}\label{lem:maximum principle}
Let
\[
\Lo=\sum_{i,j}a_{ij}(\vec x)\partial_{x_ix_j}^2+\sum_{i}b_i(\vec x)\partial_{x_i}
\]
be a uniformly elliptic operator on
 a connected, bounded, open domain $\Omega$.  Suppose that $w\in C^2(\Omega)\cap C(\bar{\Omega})$
satisfies $\Lo w\leqslant 0$ in $\Omega$ and $\partial_{\bn}w\leqslant 0$ on $\Gamma_N\subset\partial\Omega$,
where $\bn$ is the outer unit normal to $\Omega$ on the boundary $\partial\Omega$.
Also assume that $\Omega$ satisfies the interior ball condition at every $\vec x\in\Gamma_N$.
Let $\Gamma_D=\partial\Omega\setminus\Gamma_N$. Then
\[
\max_{\bar{\Omega}}w=\sup_{\Gamma_D}w.
\]
\end{lemma}
\begin{proof}
It suffices to show that $w$ attains its maximum over $\bar{\Omega}$ on $\Gamma_D$.
Let us assume $w$ is not constant within $\Omega$, as otherwise the proof is trivial.
Then the strong maximum principle \cite{Evans:1998} states that
$w$ cannot attain its maximum  over $\bar{\Omega}$ at any interior point.
Furthermore,  $w$ cannot attain its maximum  over $\bar{\Omega}$ at any boundary point $\vec x_0\in\Gamma_N$ either,
since otherwise Hopf's lemma \cite{Evans:1998} would imply $\partial_{\bn}w(\vec x_0)>0$, which contradicts the assumption $\partial_{\bn}w\leqslant 0$ on $\Gamma_N$.
Hence  $w$ has to attain its maximum over $\bar{\Omega}$ on $\Gamma_D$.
\end{proof}
The proof of Theorem \ref{thm:error estimate} relies on the following lemma, which is a consequence of  the above modified version of the maximum principle.
\begin{lemma}\label{lem:estimate}
For any $\eps>0$,
let $\Lo=-\partial_x^2-\eps^2\partial_y^2$.
Suppose that $u\in C^2(\dom)\cap C(\bar{\dom})$ satisfies
\[
\begin{cases}
 \Lo u =g & \quad \text{in } \dom ,\\
 \partial_x u(0,y)=\partial_x u(1,y)=0,  \quad &0\leq y\leq 1,\\
  u(x,0)=\phi_0(x),   \quad  u(x,1)=\phi_1(x),  &0\leq x\leq 1.
\end{cases}
\]
Then
\[
\|u\|_{\infty}\leqslant \Phi+\frac G2,
\]
where
\[
G=\sup_{(x,y)\in \dom} \lvert g(x,y)\rvert, \quad \Phi=\sup_{{0\leqslant x\leqslant 1}\atop{i=0,1}} \lvert \phi_i(x)\rvert.
\]
\end{lemma}
\begin{proof}
Let $w(x,y)=u(x,y)-v(x)$, where
\[
v(x)=\Phi+\frac{G}{2}(1-x)^2.
\]
Then it is easy to check that

\begin{align*}
& \Lo w =g-G\leqslant 0  \quad \text{in } \dom ,\\
& \partial_x w(x,y)=\partial_x u(x,y)+G(1-x)=\begin{cases} G\geqslant 0, & x=0,\\
0, & x=1,
\end{cases}\\
&  w(x,0)\leqslant \phi_0(x)-\Phi\leqslant 0,   \quad  w(x,1)\leqslant \phi_1(x)-\Phi\leqslant 0.
\end{align*}
From the modified version of the maximum principle (Lemma \ref{lem:maximum principle}), we conclude that $w(x,y)\leqslant 0$ in $\dom$,
thus
\[
u(x,y)\leqslant \Phi+\frac{G}{2}(1-x)^2\leqslant \Phi+\frac{G}{2}, \quad   \text{in } \dom.
\]
Applying the above argument to $-u$ yields
\[
-u(x,y)\leqslant \Phi+\frac{G}{2}, \quad   \text{in } \dom.
\]
Then we obtain the desired inequality.
\end{proof}
Now we give the proof of Theorem \ref{thm:error estimate}.
\begin{proof}[Proof of Theorem \ref{thm:error estimate}]
Direct calculation shows that
$r_{2n}$ satisfies
\[
\begin{cases}
 \Lo r_{2n} =\Od(\eps^{2(n+1)}) & \quad \text{in } \dom ,\\
 \partial_x r_{2n}(0,y)=\partial_x r_{2n}(1,y)=0,  \quad &0\leq y\leq 1,\\
r_{2n}(x,0)=\Od({\re}^{-\pi/\eps}),   \quad  r_{2n}(x,1)=\Od({\re}^{-\pi/\eps}),  &0\leq x\leq 1,
\end{cases}
\]
where $\Lo$ is the same as in Lemma \ref{lem:estimate}.
Then the asserted error estimates follow from Lemma \ref{lem:estimate}.
\end{proof}

\section{Numerical Example}
\label{sec:num}
We choose the source term
\[
f(x,y)=\sin(\pi(x^2+y^2))
\]
and the Dirichlet boundary conditions
\[
\phi_0(x)=\cos(\pi x), \quad \phi_1(x)=16x^2(x-1)^2.
\]
Note that the Dirichlet boundary conditions should be consistent with the Neumann boundary conditions in \eqref{eqn:PDE},
which dictates the compatibility conditions
\[
\phi_0'(0)=\phi_0'(1)=\phi_1'(0)=\phi_1'(1)=0.
\]
Clearly, these compatibility conditions are satisfied in this example.
\begin{center}
\begin{figure}[htbp]
\subfigure[The contour plot  of $u_\eps$]{
\includegraphics[width=0.475\textwidth]{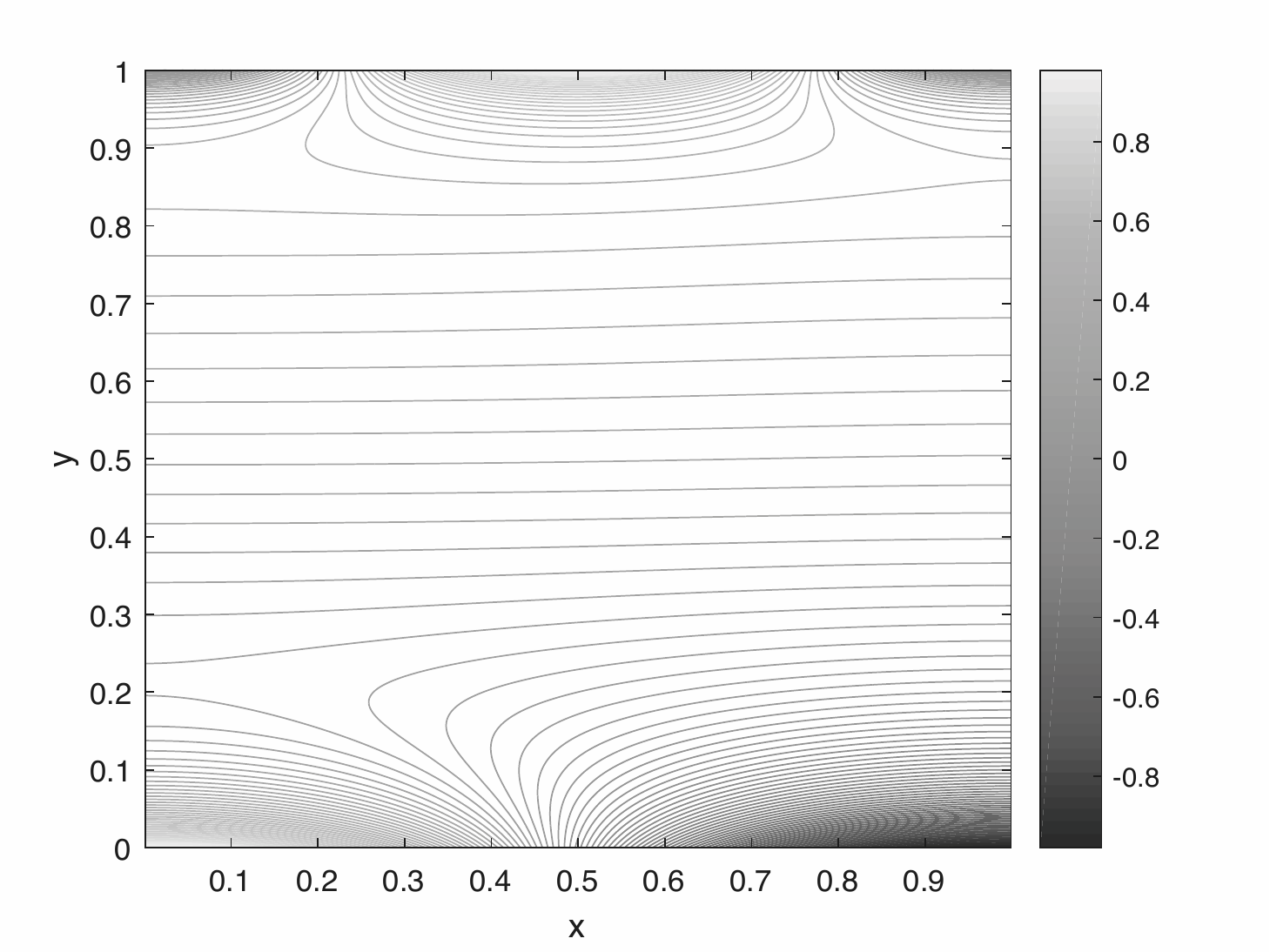}
}
\hfill
\subfigure[The contour plot  of $\bu$]{
\includegraphics[width=0.475\textwidth]{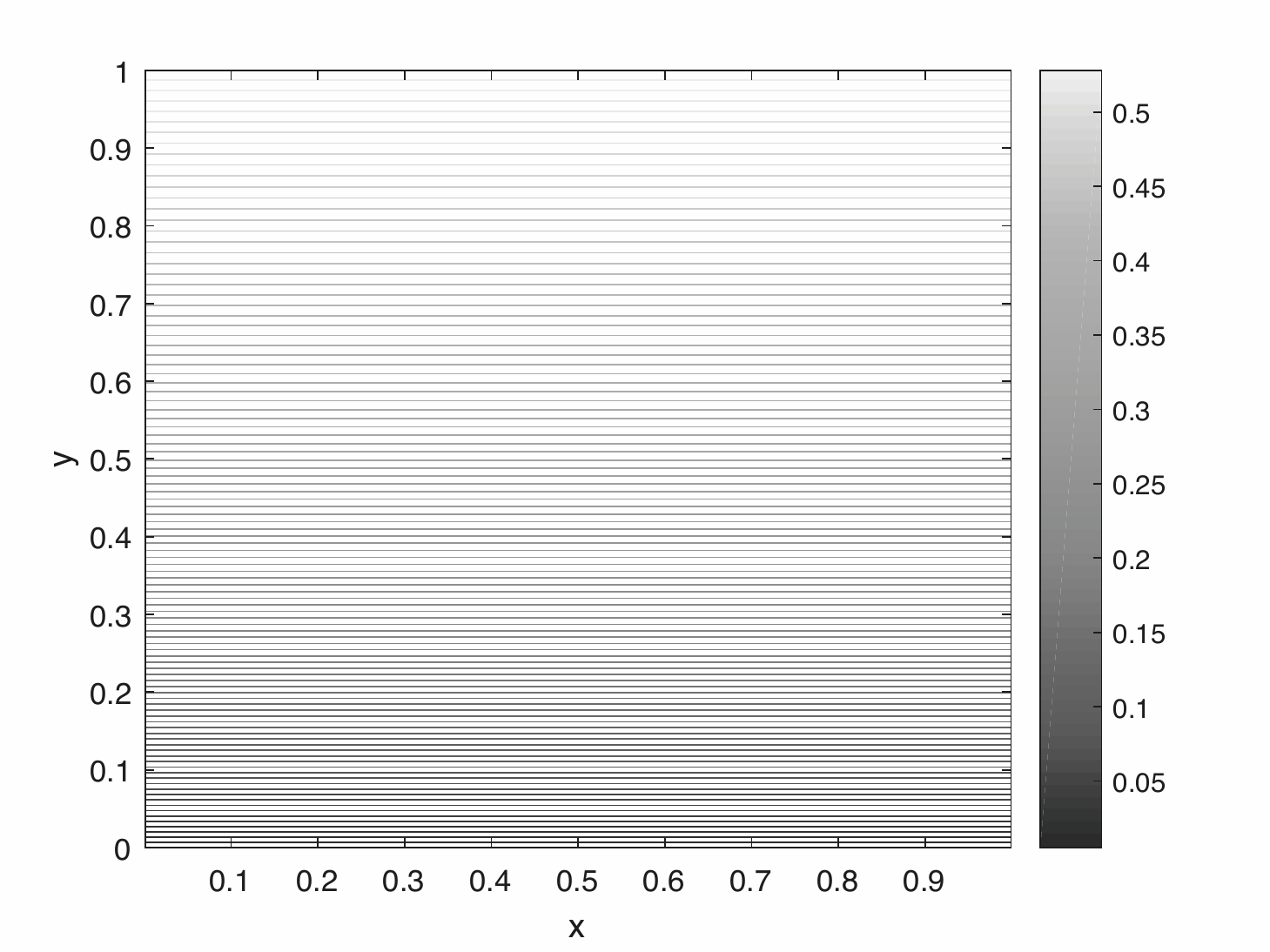}
}
\vfill
\subfigure[The contour plot  of $u^{[0]}$]{
\includegraphics[width=0.475\textwidth]{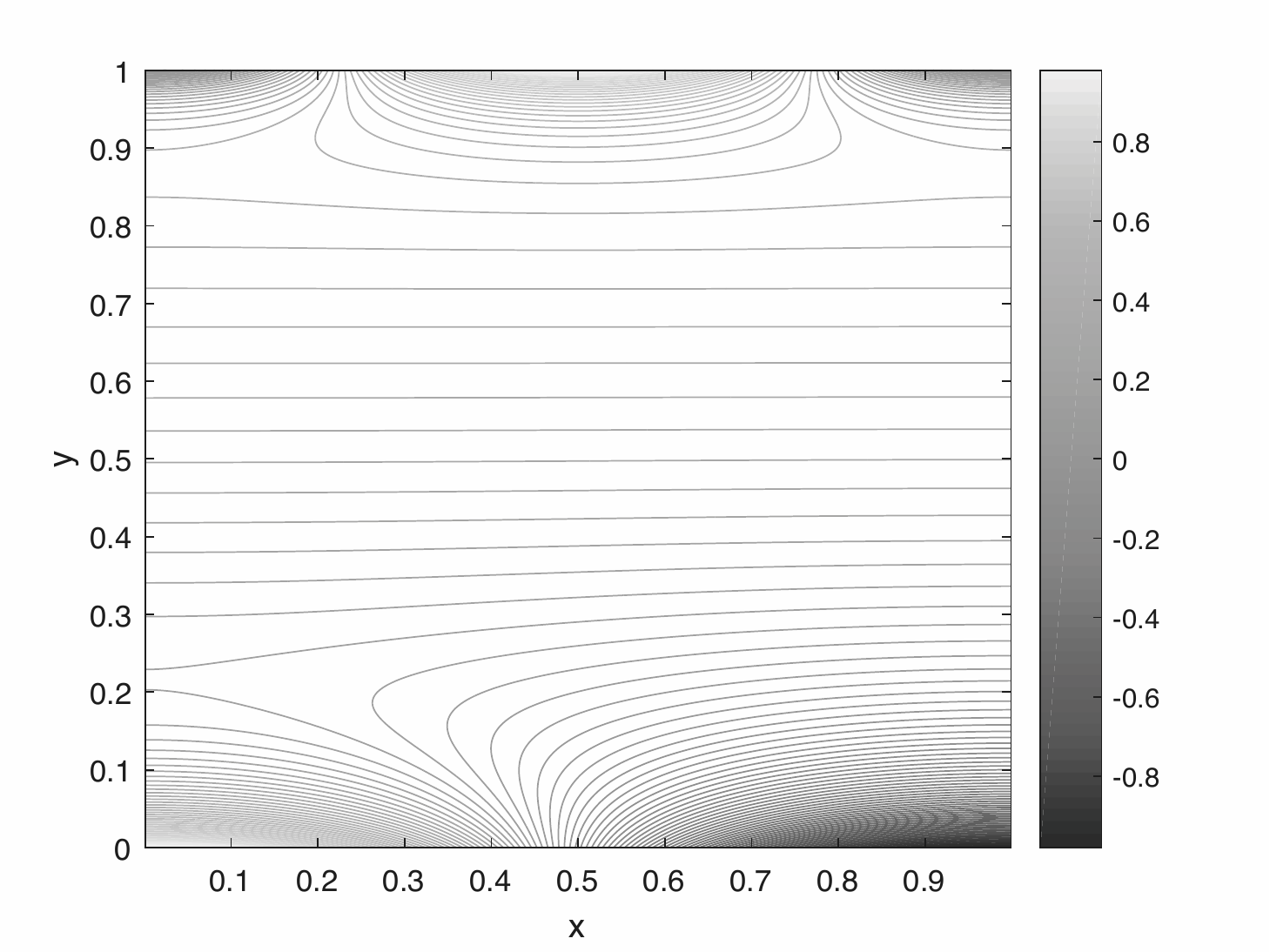}
}\hfill
\subfigure[The contour plot  of $u^{[2]}$]{
\includegraphics[width=0.475\textwidth]{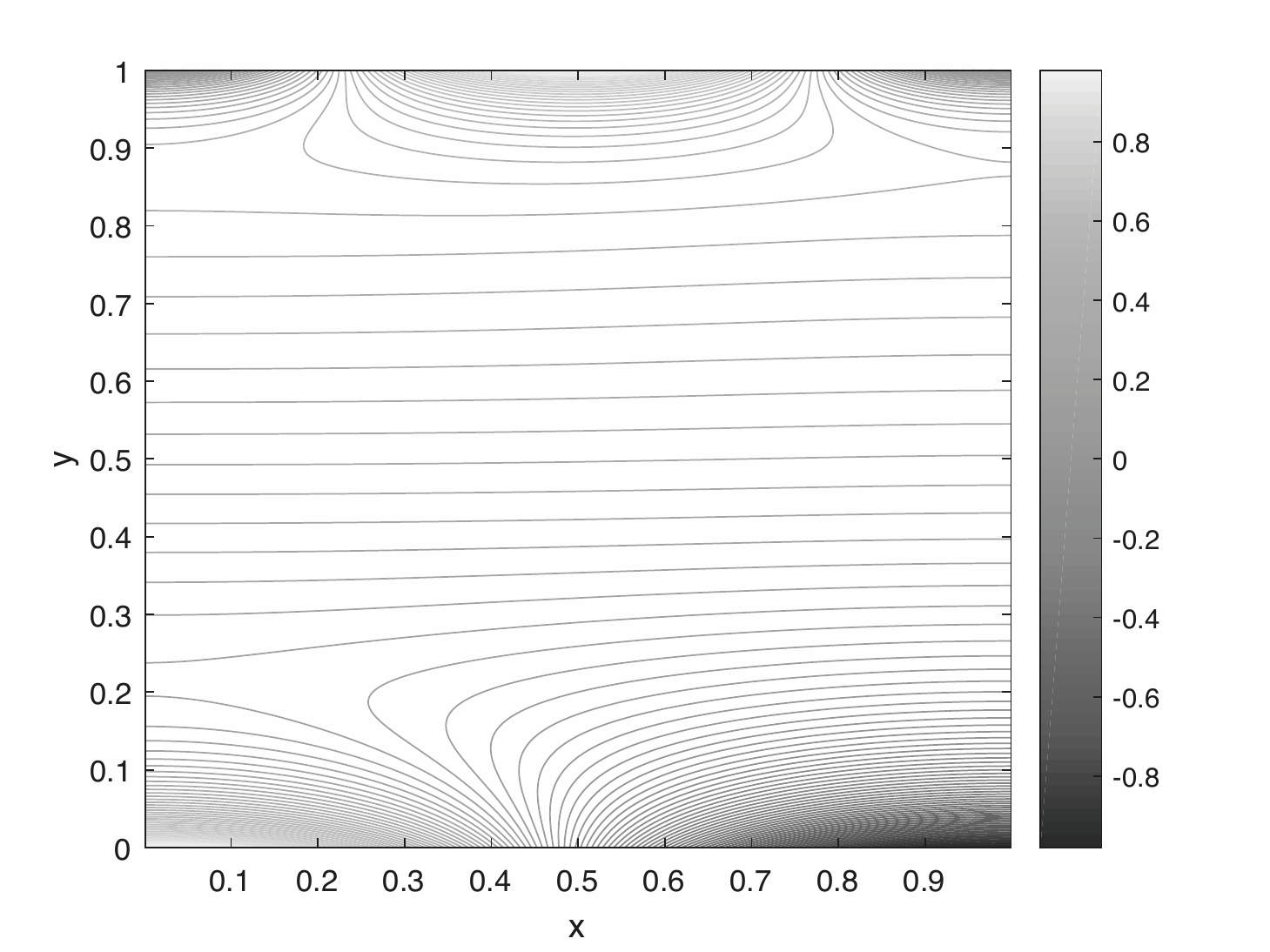}
}
\caption{ The contour plots  of the numerical solutions at $\eps^2=0.05$.
(a) the exact solution $u_\eps$;
(b) the mean part $\bu$; (c) the  asymptotic approximation   $u^{[0]}$ and (d) the asymptotic approximation $u^{[2]}$.
 }\label{fig:ueps}
\end{figure}
\end{center}

We choose a grid for which the meshpoints
are half-integered in the $x$-direction and integered in the $y$-direction, that is,
\[
x_i = (i - \dfrac 12)\Delta x,
\quad y_j = (j - 1)\Delta y,
\]
where $\Delta x =1/N$, $\Delta y=1/M$,
and $i = 1,2,\cdots,N$, $j = 1,2,\cdots, M+1$.
Then the solution $u_\eps$ to \eqref{eqn:PDE} is solved numerically by the standard  five-point  finite difference method.
 Figure \ref{fig:ueps} shows respectively the contour plots  of the numerical solutions
 with a very fine mesh size $M\times N=512\times 2048$  to  the exact solution $u_\eps$, the mean part $\bu$,   the first order asymptotic approximation
 $u^{[0]}$, and the higher order asymptotic approximation $u^{[2]}$
 for $\eps^2=0.05$.
   It is  observed that there exist two boundary layers near the boundaries $y=0$ and $y=1$ respectively,
   and the thickness of each boundary layer is roughly $\eps=\sqrt{0.05}\approx0.22$. This is consistent with the result of the asymptotic analysis
 in last section.
 Clearly, as shown in the subfigure (b), the mean solution $\bu$ fails to capture the leading order solution inside these two boundary layers.
The first order approximation $u^{[0]}$ is very close to the true solution while
 the next order   approximation $u^{[2]}$ is  almost identical to the true solution, as indicated  from the subfigures (c) and (d).
\begin{center}
\begin{table}[htbp]
\begin{tabular}[htbp]{|c|c|c|}
\hline
$\eps^2$   &  $\|r_0\|_{\infty}$ &    $\|r_2\|_{\infty}$ \\\hline
0.001     &    $1.0533E-04$        &   $5.2834E-07$   \\\hline
0.005     &   $5.2222E-04$       &    $7.3587E-06$ \\\hline
0.01      &   $1.0335E-03$      &   $2.8475E-05$     \\\hline
0.05     &    $4.7441E-03$      &    $5.7746E-04$   \\\hline
0.1      &   $8.6241E-03$       &   $1.9240E-03$  \\\hline
\end{tabular}\\[5pt]
\caption{The discrete $L_\infty$ norms of the errors $r_0$ and $r_2$ for different values of $\eps$.}
\label{tab:err}
\end{table}
\end{center}
To further validate the approximation order, we compute the errors in  $L_\infty$ norm.
Table \ref{tab:err} shows the discrete $L_\infty$ norms of the
errors $r_0$ and $r_2$ for different values of $\eps$,
and Figure \ref{fig:err} is the convergence  plot of  these  errors versus $\eps^2$ on logarithmic scales.
The best fitting straight lines through each set of data points are also plotted in Figure \ref{fig:err}.
From the equations of these fitting lines,
we observe that the orders of the approximation errors $r_0$
and $r_2$ are roughly $\Od(\eps^2)$ and $\Od(\eps^4)$ respectively.
These numerical results confirm the theoretical
assertion in Theorem \ref{thm:error estimate}.

\begin{center}
\begin{figure}[htbp]
\includegraphics[width=0.5\textwidth]{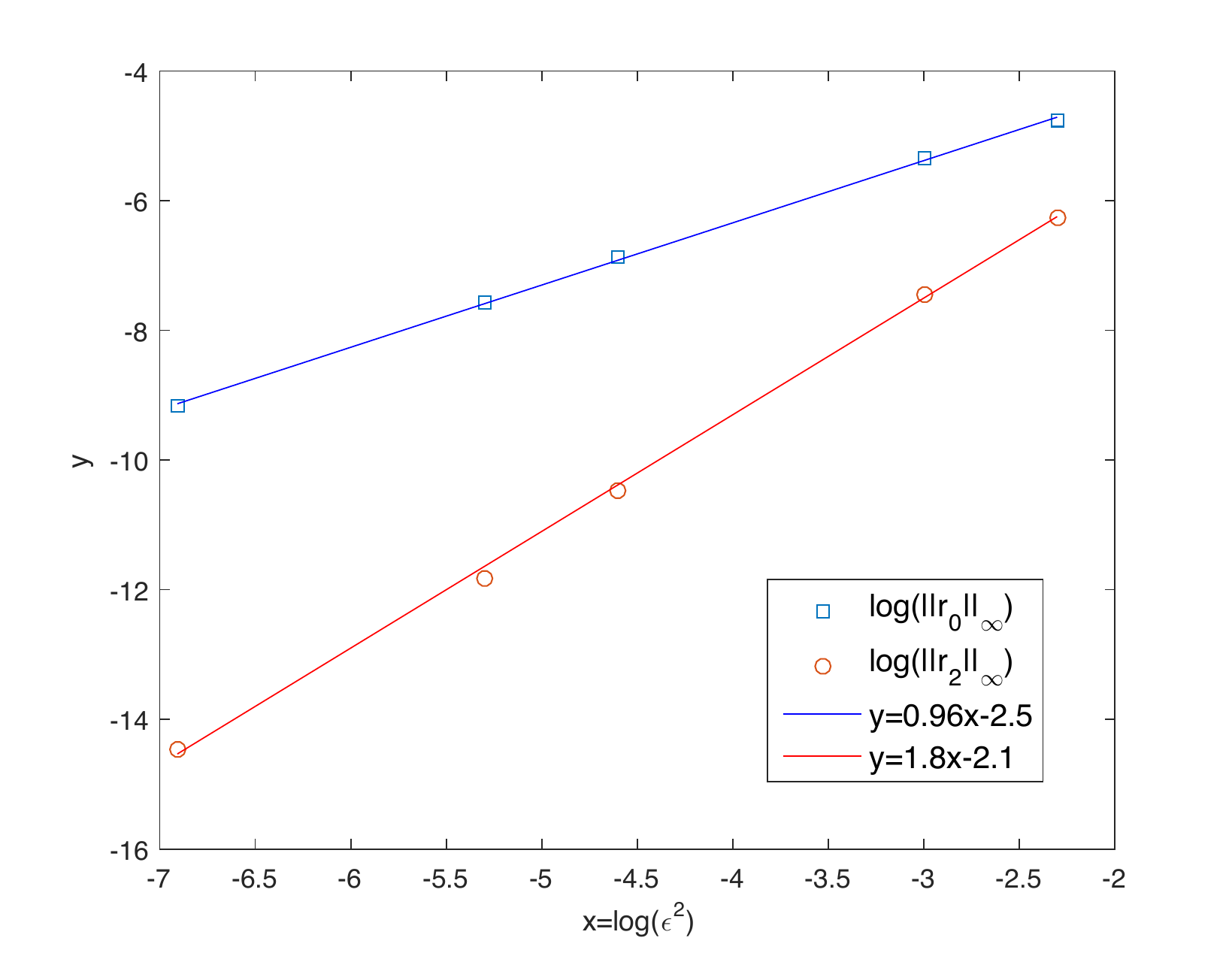}
\caption{The plot on logarithmic scales of the discrete $L_\infty$ norms of the errors $r_0$ (in  the square-shaped markers) and $r_2$ (in  the circle-shaped markers) versus $\eps^2$.
The  best fitting straight lines through each set of data points are also shown.
Their equations are indicated in the legend.
}
\label{fig:err}
\end{figure}
\end{center}

\section{Conclusion}

We have presented a formal expansion of the solution to the strongly anisotropic
diffusion equation \eqref{eqn:PDE} in the simple rectangular domain.
Our expansions take into account the boundary layers near the Dirichlet boundaries.
We proved the rigorous convergence of the formal expansion by the maximum principle.
In theory,  our method  can be generalized to   complicated cases
where the field line is not simply the $x$-axis direction, but in the form of curves.
For such cases, the analysis of the boundary layers is more difficult.
Nevertheless, we expect that our analysis here would  provide some ideas of
constructing new efficient numerical methods which
 incorporate  both the boundary layer effect due to the Dirichlet boundary condition and
 the ill-posedness due to the Neumann or periodic boundary condition in the strongly
 anisotropic diffusion limit.

\bibliographystyle{amsplain}
\bibliography{reference}

\end{document}